\documentclass[11pt,epic]{article}
\usepackage[dvips]{graphicx}
\usepackage[usenames,dvipsnames]{color}
\usepackage{authblk}

\usepackage[ruled,lined,algonl,boxed]{algorithm2e}
\usepackage{amsmath}

\begin{document}
\newcommand{\up}{\vspace*{-0.136cm}}
\newcommand{\qed}{\hfill$\rule{.05in}{.1in}$\vspace{.3cm}}
\newcommand{\pf}{\noindent{\bf Proof: }}
\newtheorem{thm}{Theorem}
\newtheorem{lem}{Lemma}
\newtheorem{prop}{Proposition}
\newtheorem{prob}{Problem}
\newtheorem{ex}{Example}
\newtheorem{cor}{Corollary}
\newtheorem{conj}{Conjecture}
\newtheorem{cl}{Claim}
\newtheorem{df}{Definition}
\newtheorem{rem}{Remark}
\newcommand{\beq}{\begin{equation}}
\newcommand{\eeq}{\end{equation}}
\newcommand{\<}[1]{\left\langle{#1}\right\rangle}
\newcommand{\be}{\begin{enumerate}}
\newcommand{\ee}{\end{enumerate}}
\newcommand{\Bul}{\mbox{$\bullet$ } }
\newcommand{\al}{\alpha}
\newcommand{\ep}{\epsilon}
\newcommand{\si}{\sigma}
\newcommand{\om}{\omega}
\newcommand{\la}{\lambda}
\newcommand{\La}{\Lambda}
\newcommand{\Ga}{\Gamma}
\newcommand{\ga}{\gamma}
\newcommand{\im}{\Rightarrow}
\newcommand{\2}{\vspace{.2cm}}
\newcommand{\es}{\emptyset}

\title{\bf Approximation algorithms and ratios for multiple domination in graphs}
\author[1]{Lukas Dijkstra}
\author[2]{Vadim Zverovich\footnote{Corresponding author (e-mail: vadim.zverovich@uwe.ac.uk)}}
\author[1]{Andrei Gagarin}
\affil[1]{\small School of Mathematics, Cardiff University, UK}
\affil[2]{\small Mathematics and Statistics Research Group, University of the West of England, UK}

\date{16 December 2025}
\maketitle

\begin{abstract}
We analyse approximation algorithms (greedy heuristics) for the classical domination number and two multiple domination numbers in simple graphs. 
First, we present a short self-contained proof of the known result that the minimum domination problem in any graph $G$ with maximum degree $\Delta$ can be solved within the approximation ratio of ${\ln(\Delta+1)+1}$. 
The proof is based on an analysis of a simple greedy heuristic.
Then, by analysing more advanced greedy heuristic techniques and using ideas from our self-contained proof for the classical domination number, we fix a gap in the existing proof of a similar result for the $k$-tuple domination number.
That is, we prove that the minimum $k$-tuple domination problem indeed can be approximated within the ratio of $\ln(\Delta+1)+1$. 
The proof of this result is self-contained, direct, and much shorter than the existing proof, which contains the gap.
Finally, we show that the known approximation ratio of $\ln(2\Delta)+1$ for the minimum $k$-domination problem can be improved to a better ratio.
\end{abstract}

\section{Introduction}

Domination in graphs has been widely studied \cite{hay1} 
and motivated by real-life applications in various areas \cite{zve1}.
There are many interesting results devoted to the multiple domination parameters (e.g. see \cite{fav2,gag1,LC2013, rau1}).

We consider finite undirected (simple) graphs without loops or 
multiple edges. Given a graph $G=(V,E)$ of order $n$, 
we denote by $\delta(G)$
and $\Delta(G)$ the minimum and maximum degrees of vertices of
$G$, respectively. 
When $G$ is clear from the context, we put $\delta=\delta(G)$ and $\Delta=\Delta(G)$.
The neighbourhood of a vertex $v\in V$ is denoted by $N(v)$. 
Then, for a set of vertices $X\subseteq V$, $N(X)=\cup_{v\in X} {N(v)} $ and $ N[X]=N(X)\cup X$.

A vertex $v\in V$ is \emph{dominated} by a set of vertices $X\subseteq V$ if $v$ is in $X$ or adjacent to a vertex in $X$; 
otherwise, $v$ is {\it not dominated} by $X$ in $G$. 
A set $X\subseteq V$ is called a {\it dominating set} of $G$ if every vertex not in
$X$ is adjacent to a vertex in $X$. The minimum cardinality of a
dominating set of $G$ is the {\it domination number} $\gamma(G)$.
Given a positive integer $k$, a set $X\subseteq V$ is called a {\it $k$-dominating set} of $G$ if every vertex not
in $X$ has at least $k$ neighbours in $X$. The minimum cardinality
of a $k$-dominating set of $G$ is the \emph{$k$-domination number} $\gamma_k(G)$. 
A set $X\subseteq V$ is called a {\it $k$-tuple dominating set} of $G$ if, for
every vertex $v\in V$, $|N[v]\cap X|\ge k$. The minimum
cardinality of a $k$-tuple dominating set of $G$ is the {\it
$k$-tuple domination number} $\gamma_{\times k}(G)$. The $k$-tuple
domination number is only defined for graphs with $\delta\ge k-1$. 
It is easy to see that $\gamma(G)=\gamma_{1}(G)=\gamma_{\times1}(G)$, 
$\,\gamma_k(G)\le \gamma_{\times k}(G)$, $\,\gamma_{k}(G)\le \gamma_{k'}(G)$, and
$\,\gamma_{\times k}(G)\le \gamma_{\times k'}(G)$ for $k\le k'$ (where $k'$ is an integer).

The problems of finding $\gamma(G)$, $\gamma_{k}(G)$, and $\gamma_{\times k}(G)$ and the corresponding sets of vertices in a graph are known to be NP-complete \cite{GJ1979,LC2003,LC2013}. 
Moreover, these problems are known to be APX-hard \cite{AK2000,kla1} and, in general, not fixed parameter tractable \cite{DF1999}.
Therefore, the problems require usage of efficient heuristic methods. 
In this paper, those problems are addressed from the point of view of their approximability analysis. 
We provide short self-contained proofs of two previously known results and correct a mistake in one of them.
We also improve the known approximation ratio for the minimum $k$-domination problem. 
 
For a positive integer $x$, let $H(x)$ denote the \emph{harmonic
number} of rank $x$, i.e. 
$$
H(x)=1+ {1 \over 2}+ \dots +{1 \over x}, \quad
\mbox{and}\quad H(0)=0.
$$ 
The following technical result is well known:

\begin{lem} 
\label{lemmaH}
If $x, y$ are integers, $x\ge 1$, and $x\ge y\ge 0$, then  $${x-y \over x}
\le H(x)-H(y).$$ 
Moreover, $$H(x) \le \ln(x) +1.$$
\end{lem}

The paper is organized as follows. In Section \ref{sec-dom}, we provide a new simple self-contained direct proof of the approximation ratio for the domination number $\gamma(G)$.
In Section \ref{sec-k-tuple-dom}, we fix the gap in the known proof of the approximation ratio for the $k$-tuple domination number $\gamma_{\times k}(G)$ and provide a new self-contained direct proof for this ratio.
In Section \ref{sec-k-dom}, we improve a previously known approximation ratio for the $k$-domination number $\gamma_{k}(G)$ and provide a self-contained direct proof of this result.

\section{Approximation Algorithm for a Dominating Set}
\label{sec-dom}

The following straightforward greedy heuristic, due to Johnson
\cite{joh1} and described in Algorithm~\ref{alg:dom}, is an approximation algorithm for the minimum
dominating set problem in a simple graph $G$ with an approximation ratio of
$\ln(\Delta+1)+1$.
The algorithm finds an ordered dominating set $D=\{d_1,d_2,...,d_{|D|}\}$, where each vertex $d_i$ is chosen in iteration $i$
of Algorithm~\ref{alg:dom}, $i=1,2,...,|D|$. 
Let $D_i^*$ denote the set of vertices in $G$ dominated by $\{d_1,d_2,...,d_i\}$ at the end of iteration $i$, 
$i=1,2,...,|D|$, i.e. 
$$D_i^*=N[\{d_1,d_2,\dots,d_i\}],\quad D_0^*=\emptyset.$$ 
Denote by $U_{i-1}[v]=N[v]-D_{i-1}^{*}$ the set of vertices in the closed neighbourhood of a vertex $v$ not dominated yet at the beginning of iteration $i$ by $\{d_1,d_2,...,d_{i-1}\}$, $i=1,2,...,|D|$.

\begin{algorithm} \label{alg:dom}
    \caption{Dominating Set (\cite{joh1})}
    \KwIn{A simple graph $G$.}
    \KwOut{A dominating set $D$ of $G$.}
    \BlankLine
\Begin{

Initialize $i=0$\; 
Initialize $D=\emptyset$\;
Initialize $D_0^{*}=\emptyset$\;
\While{ $V(G) - D_i^{*}\neq \emptyset$} {
 Put $i=i+1$\; 
 Choose $v \in V(G)-D$  maximizing the cardinality of\  
 $U_{i-1}[v] =  N[v]-D_{i-1}^{*}$\;
 Put $d_i=v$\;
 Put $D=D\cup \{d_i\}$\;
 Put $D_{i}^{*}=N[D]$; \tcc*[f]{set of vertices dominated by current set $D$}

}
\Return $D$; \tcc*[f]{$D$ is a dominating set in $G$} }

\end{algorithm}

Let $C[d_i]=U_{i-1}[d_i]$ be the set of vertices only dominated by the inclusion of $d_i$ in the set $D$ in iteration $i$, $i=1,2,..,|D|$. 
For a vertex $v \in V(G)$, let $s=s(v)$ be the index such that
$$v \in C[d_s]=U_{s-1}[d_s] = N[d_s]- D_{s-1}^{*},$$ i.e. $s$
is the index of the first vertex $d_s\in N[v]\cap D$ that dominates $v$ in the iteration.

Both Algorithm \ref{alg:dom} and Theorem \ref{thm-dom} are
formulated in \cite{joh1} in terms of the set cover problem. 
We formulate them here directly in terms of the dominating set problem, and our proof of the theorem presented below is much shorter than the original proof in \cite{joh1}.

\begin{thm}[\cite{joh1}] \label{thm-dom}
The minimum dominating set problem in any graph $G$ of maximum
degree $\Delta$ can be solved within an approximation ratio of
$\,\ln(\Delta+1)+1$ in polynomial time.
\end{thm}

\pf First, we apply the greedy algorithm to find a dominating set $D$
of $G$. It is easy to see that the algorithm is polynomial in time and space ($O(nm)$ in time). 
Let $D^{\min}$ be a dominating set of cardinality $\ga(G)$. 
Let us prove that $|D| \le |D^{\min}|(\ln(\Delta+1)+1)$. 

It is obvious that
$\sum_{x\in X}{1\over |X|}=1$ for any non-empty set $X.$ 
Since we have a partition 
$V(G) =\bigcup_{i=1}^{|D|}C[d_i]$,\   
$C[d_i]\cap C[d_j]=\emptyset$\  for any $i,j$,\  $1\le i<j\le |D|$, 
each vertex $v \in V(G)$ is in exactly one
set $C[d_i]$, $i=1,2,...,|D|$.
We obtain
\begin{eqnarray*}
|D| =  \sum_{i=1}^{|D|}\,\,{\sum_{v \in C[d_i]}}{1
\over |C[d_i]|} \;\,
 = \sum_{v \in V(G)} {1 \over |C[d_{s}]|}.
\end{eqnarray*}
The value $c_v = {1 \over {|C[d_s]|}}$ is called the
{\it cost} of the vertex $v$. 
We have
\begin{eqnarray} \label{1}
|D| \;= \sum_{v \in V(G)}c_v \;\,\le \sum_{u \in D^{\min}}\sum_{v \in
N[u]}c_v
\end{eqnarray}
because $c_v$ is counted exactly once for each $v \in V(G)$ on the
left-hand side of this inequality, whereas it is counted at least
once on the right-hand side, since $D^{\min}$ is a dominating
set.

Let us now consider the following non-increasing sequence:
$r_i=|U_{i}[u]|$, where $U_i[u]=N[u]-D^*_i$ in Algorithm~\ref{alg:dom}, 
$i=0,1,...,m$, and $m$ is the smallest index 
such that $r_m=0$ (i.e. $U_m[u]=\emptyset$). Note that the cost $c_v$ of $v \in N[u]$ is ${1
\over |C[d_{s}]|}$, and exactly $r_{i-1} - r_i$ vertices of
$N[u]$ are dominated for the first time in iteration $i$ of Algorithm~\ref{alg:dom} with the cost
${1\over |C[d_i]|}$. Therefore,
\begin{eqnarray*}
\sum_{v \in N[u]}c_v = \sum_{i=1}^{m}{r_{i-1} - r_i \over
|C[d_i]|} \le \sum_{i=1}^{m} {r_{i-1} - r_i \over
|U_{i-1}[u]|} = \sum_{i=1}^{m}{r_{i-1} - r_i \over
r_{i-1}}.
\end{eqnarray*}
Here, $|C[d_i]|=|U_{i-1}[d_i]| \ge |U_{i-1}[u]|$
because of the choice of the vertex $d_i$ in the iteration. 
By Lemma \ref{lemmaH},
\begin{eqnarray*}
\sum_{v \in N[u]}c_v \le \sum_{i=1}^{m}\Big(H(r_{i-1}) -
H(r_i)\Big) = H(r_0) - H(r_m) = H(|N[u]|) \le \ln(\Delta+1)+1.
\end{eqnarray*}
Using (\ref{1}), we obtain $|D| \le |D^{\min}|(\ln(\Delta+1)+1)$, as
required.  \qed

Feige \cite{fei1} proved that the minimum dominating set problem is
not approximable within $(1-\epsilon)\ln n$ for any $\epsilon>0$,
unless NP $\subseteq$ DTIME$(n^{\log\log n})$.
Therefore, the approximation ratio of $\ln(\Delta+1)+1$ is the best possible 
(asymptotically, in the worst case).


\section{Approximation Algorithm for a $k$-Tuple Dominating Set}
\label{sec-k-tuple-dom}

Klasing and Laforest \cite{kla1} genralized Algorithm~\ref{alg:dom} and Theorem~\ref{thm-dom} for the $k$-tuple
domination problem. In particular, they proved an approximation ratio of $\ln(\Delta+1)+1$ for $k$-tuple domination.
However, their proof is flawed for values of $k$ greater than 1.
More precisely,  in \cite{kla1}
there is a gap in the proof of Lemma 6 for $k>1$ that cannot be fixed by a minor modification.
The greater-than-or-equal-to inequality in the middle of the proof of Lemma 6 in general is not true for $k>1$ 
because `the maximum choice of the algorithm at each step' 
is done in \cite{kla1} for a restricted number of sets $S$, not for all. 
In this section, we fix the gap in \cite{kla1} and show that the approximation ratio is indeed $\ln(\Delta+1)+1$. 
Note that the proof in \cite{kla1} is based on six lemmas, whereas our proof presented 
in this section (see Theorem~\ref{aath2} below) is much shorter, more direct, and self-contained.

Recall that, in the proof of Theorem \ref{thm-dom} above, we used the inequality $|U_{i-1}[d_i]| \geq |U_{i-1}[u]|$, $i=1,2,...,m$. 
This was justified by the fact that $d_i$ is chosen by the algorithm as the vertex to be added to $D$ in iteration $i$
and, therefore, must have the largest cardinality of undominated closed neighbourhood after iteration $i-1$. 
However, this only works under the assumption that $u$ is also one of the vertices examined by the algorithm in this particular iteration.
When $k \geq 2$, if $u$ has already been included in $D$ before the $i$th iteration,
it cannot be the next vertex selected for $D$ even if its closed neighbourhood that is not $k$-tuple dominated
is larger than that of $d_i$.

In the case $k = 1$, having $u$ added to $D$ in a previous iteration means $U_{i-1}[u] = \emptyset$ in all the later iterations.
Therefore, the above situation cannot occur.
In the case $k \geq 2$, this scenario is possible.
For example, if the vertex $u$ has a closed neighbourhood strictly larger than any other vertex in the graph has,
then the algorithm will select it as the vertex $d_1$ in the first iteration.
However, as $k \geq 2$, none of the vertices in the closed neighbourhood of $u$ will be 
fully $k$-tuple dominated in the second iteration.
Therefore, $u$ once again has a non-$k$-tuple-dominated closed neighbourhood whose size is
strictly larger than that of every other vertex in the second iteration.
Since the algorithm now must select a vertex other than $u$ for $d_2$, we have $|U_{1}[d_2]| < |U_{1}[u]|$.
This would be a contradiction to the above mentioned inequality used in the proof of Theorem \ref{thm-dom} in the case $k\ge 2$. 
Therefore, the generalization of the proof of Theorem \ref{thm-dom} for $k\ge 2$
is not straightforward. 

A vertex $v\in V(G)$ is \emph{$k$-tuple dominated} by a set of vertices $X\subseteq V(G)$ if $|N[v]\cap X| \ge k$. 
Otherwise, $v$ is \emph{not $k$-tuple dominated} by $X$ in $G$. The set of vertices $k$-tuple dominated by $X$ in $G$ is denoted by $C_{\times k}[X]$.
Algorithm~\ref{BGT} described in the pseudocode below is a greedy heuristic, due to Klasing and Laforest \cite{kla1}. 
Similar to Algorithm~\ref{alg:dom}, Algorithm~\ref{BGT} returns an ordered $k$-tuple dominating set $D=\{d_1,d_2,...,d_{|D|}\}$, where each vertex $d_i$ is chosen in iteration $i$, $i=1,2,...,|D|$.
 
\vspace{5mm}
\begin{algorithm}[H]
    \caption{$k$-Tuple Dominating Set (\cite{kla1})}
    \label{BGT}
    \KwIn{A graph $G=(V,E)$, an integer $k$, $1\le k \le \delta(G)+1$.}
    \KwOut{A $k$-tuple dominating set $D$ of $G$.}
    \BlankLine
    
    \Begin{
    	Initialise $i = 0$;\\
        Initialise $D = \emptyset$;\\
        \ForEach{$v \in V$} {
            Compute $\tau(v) = |N[v] - C_{\times k}[D]|$;
        }

        \While{$C_{\times k}[D] \neq V$} {
        	    Put $i=i+1$;\\
            Select a vertex $u \in U = \mathrm{arg\ max}_{v \in V - D}\ \,\tau(v)$;\\
            Put $d_i=u$;\\
            Put $D = D \cup \{d_i\}$;\\
            Update the values of $\tau(v)$ for $v \in V - D$;
        }
        
    \Return $D$;
    }
\end{algorithm}
\vspace{5mm}

A generalization of the proof of Theorem~\ref{thm-dom} for
$k$-tuple domination, which resolves a flaw in the corresponding proof in \cite{kla1}, is presented in Theorem~\ref{aath2}.
The proof is tailored for 
$k$-tuple domination and does not rely on concepts from the more general set cover problem used in \cite{kla1}.
It is based on a generalization of Algorithm~\ref{alg:dom} to 
$k$-tuple domination, presented in Algorithm~\ref{BGT}.

\begin{thm} \label{aath2}
The minimum $k$-tuple dominating set problem in a graph $G$ of 
maximum degree $\Delta$ can be solved within an approximation
ratio of
$\,\ln(\Delta+1)+1$ in polynomial time.
\end{thm}

\pf
Similar to the proof of Theorem~\ref{thm-dom}, let $D$ be the $k$-tuple dominating set constructed by Algorithm \ref{BGT},
and let $d_i$ be the vertex selected by Algorithm~\ref{BGT} to be included in $D$ in the $i$th iteration, $i=1,2,...,|D|$. 
Also, denote by $D_i = \{d_1,d_2,..., d_i\}$, $i=1,2,...,|D|$, the set of the first $i$ vertices selected by the algorithm, and put $D_0 = \emptyset$.
The closed neighbourhood of $v \in V$ that is not $k$-tuple dominated
after iteration $i$ is now defined as $U_i[v] = N[v] - C_{\times k}[D_i]$.
Thus, only vertices that are fully $k$-tuple dominated are removed from $U_i[v]$,
as opposed to $U_i[v] = N[v] - N[D_i]$, which was used in the proof of Theorem \ref{thm-dom}.

Since $k$ vertices are required to $k$-tuple dominate $v$, we define $\bar{s}(v)$ as 
a vector whose entries $s_j(v)$, for $1 \leq j \leq k$,
represent the first $k$ iterations in which a vertex in the closed neighbourhood of $v$ 
is selected for inclusion in $D$.
Let $D_v = \{d_{s_1(v)},d_{s_2(v)},..., d_{s_k(v)}\} \subseteq V$ be the set
consisting of the $k$ members of the $k$-tuple dominating set represented by the indices in $\bar{s}(v)$.

The cost of $v$ is also redefined.
Instead of having a single cost, $v$ now has a separate cost for each member of its closed neighbourhood.
For $w \in N[v]$, the \textit{cost of $v$ with respect to $w$} is defined as follows:
$$
c_v(w) = \begin{cases} \displaystyle
\frac{1}{|U_{s_j(v)-1}[d_{s_j(v)}]|} &\text{if $w = d_{s_j(v)}$ for some $j \in \{1,2,\dots,k\};$} \\ \displaystyle
\frac{1}{|U_{s_k(v)-1}[d_{s_k(v)}]|} &\text{otherwise.}
\end{cases}
$$
This gives the following result for the cardinality of the $k$-tuple dominating set found by the algorithm.
\begin{lem} \label{tdaalem1}
We have $\displaystyle |D| = \sum_{v \in V} \sum_{x \in D_v} c_v(x)$.
\end{lem}
\pf
Note that in the expressions below a summation over $D_v$ is equivalent to taking the sum over all vertices $v$ such that $d_i \in D_v$:
\begin{align*}
\sum_{v \in V} \sum_{x \in D_v} c_v(x) &= \sum_{i = 1}^{|D|} \sum_{v:\,d_i\in D_v} c_v(d_i) 
                                       = \sum_{i = 1}^{|D|} \sum_{v:\,d_i\in D_v} \frac{1}{|U_{i-1}[d_i]|} \\
                                       &= \sum_{i = 1}^{|D|} \sum_{v \in U_{i-1}[d_i]} \frac{1}{|U_{i-1}[d_i]|} 
                                       = \sum_{i = 1}^{|D|} 1 = |D|.
\end{align*}
\qed

Now, let $W \subseteq N[v]$ be a subset of the closed neighbourhood of $v$, and $|W| \geq k$.
The sum of costs of $v$ with respect to the vertices in $W$ will always be at least as large as the sum of costs of $v$ with respect to the vertices in $D_v$.
This is proved below.

\begin{lem} \label{tdaalem2}
For $W \subseteq N[v]$ with $|W| \geq k$, 
$$\displaystyle \sum_{x \in D_v} c_v(x) \leq \sum_{w \in W} c_v(w).$$
\end{lem}

\pf
For any $j < k$,
$$|U_{s_j(v)-1}[d_{s_j(v)}]| \geq |U_{s_j(v)-1}[d_{s_k(v)}]|$$
because $d_{s_k(v)}$ has not yet been included in the set $D$ and, therefore, 
both $d_{s_j(v)}$ and $d_{s_k(v)}$ are examined by the algorithm in iteration $s_j(v)$.
Moreover, as $|U_{i}[d_{s_k(v)}]|$ is a non-increasing, non-negative sequence with respect to $i$,
$$|U_{s_j(v)-1}[d_{s_k(v)}]| \geq |U_{s_k(v)-1}[d_{s_k(v)}]| .$$
Therefore, $c_v(d_{s_j(v)}) \leq c_v(d_{s_k(v)})$.
Additionally, $|W - D_v| \geq |D_v - W|$ as $|W| \geq k = |D_v|$, 
and for every $w \in N[v] - D_v$, the cost of $w$ is the same as that of $d_{s_k(v)}$.
Together, these three facts give us the following:
\begin{align*}
\sum_{x \in D_v} c_v(x) &= \!\!\sum_{x \in D_v \cap W} c_v(x)\;\; + \sum_{x \in D_v - W} c_v(x) 
                        \leq \sum_{x \in D_v \cap W} c_v(x)\;\; + \sum_{x \in D_v - W} c_v(d_{s_k(v)}) \\ \\
                        &\leq \!\sum_{w \in W \cap D_v}\! c_v(w)\; +\! \!\sum_{w \in W - D_v}\! c_v(d_{s_k(v)}) 
                        \, = \!\!\sum_{w \in W \cap D_v}\! c_v(w)\; +\!\! \sum_{w \in W - D_v}\! c_v(w) 
                        = \!\!\sum_{w \in W}\! c_v(w).
\end{align*}
\qed

Next, take some vertex $w \in V.$
As in the proof of Theorem \ref{thm-dom}, a bound on the sum of costs of the closed neighbourhood of $w$ will be found.
The costs will be those of the neighbours $v \in N[w]$ of $w$ with respect to $w$,
i.e.\ we are interested in $\sum_{v \in N[w]} c_v(w)$ and not in 
$\sum_{v \in N[w]} c_w(v)$.
Furthermore, the values of $r_i$ are now defined as follows:
$$
r_i = \begin{cases} \displaystyle
0 &\text{if $w = d_j$ for some $j \in \{1,2,\dots,k\}$ and $i \geq j$}; \\ \displaystyle
|U_i[w]| &\text{otherwise.}
\end{cases}
$$
This definition means that the numbers $r_i$ function the same way as in the proof of Theorem \ref{thm-dom}, 
unless $w$ is included in $D$ at some iteration, in which case $r_i$ immediately becomes equal to $0$.
However, regardless of whether this happens, $m$ is still defined to be the first index such that $r_m = 0$.
The resulting bound is as follows.

\begin{lem} \label{tdaalem3}
For all $\displaystyle w \in V,$
$$
\sum_{v \in N[w]} c_v(w) 
\leq \sum_{i = 1}^{m} \frac{r_{i-1} - r_i}{r_{i-1}}
\leq \mathrm{ln}(|N[w]|) + 1.
$$
\end{lem}

\pf
Each vertex $v \in N[w]$ will have cost $c_v(w) =  \frac{1}{|U_{i-1}[d_i]|}$ for some $i$, $i=1,2,...,|D|$. 
If $w \notin D_v$, the index used in the cost will default to $i = s_k(v)$, given by the iteration in which $v$ is 
completely $k$-tuple dominated.
However, if $w = d_j \in D_v$ for some $j$,
then the cost will instead use the index $i = j$ of the iteration in which $w$ is included in the $k$-tuple dominating set.

Note that if $w = d_j$ for some $j$, then the closed neighbourhood of $w$ that is not $k$-tuple dominated must have been non-empty after iteration $j - 1$,
and, therefore, $m = j$ by the definition of $r_i$. 
Thus, when $i < m$ and, therefore, $w \neq d_i$, 
we have $r_{i-1} - r_i = |\{v \in N[w] \mid s_k(v) = i\}|$,
which is the number of vertices completely $k$-tuple dominated in iteration $i$,
each of them with cost $c_v(w) = \frac{1}{|U_{i-1}[d_i]|}$.

For $i = m$, there are two cases.
If $d_m \neq w$, all remaining vertices $v \in U_{m-1}[w]$ are fully $k$-tuple dominated by $d_m$,
which means that they all have cost $c_v(w) = \frac{1}{|U_{m-1}[d_m]|}$.
If instead $d_m = w$, then $w \in D_v$ for all remaining vertices $v \in U_{m-1}[w]$, as they were not yet $k$-tuple dominated at iteration $m - 1$.
Therefore, by definition, all of them also have cost $c_v(w) = \frac{1}{|U_{m-1}[d_m]|}$ here.

Ultimately, in every case, $r_{i-1} - r_i$ ends up indicating the number of vertices with cost $c_v(w) = \frac{1}{|U_{i-1}[d_i]|}$
for all $i$, $i=1,2,...,|D|$.
This gives
$$\sum_{v \in N[w]} c_v(w) = \sum_{i = 1}^{m} (r_{i-1} - r_i) \frac{1}{|U_{i-1}[d_i]|} .$$

Recall that if $w$ is included in the $k$-tuple dominating set at some point, then it must be at iteration $m$.
Therefore, until at least iteration $m$, $w$ is examined by the algorithm when choosing the optimal vertex $d_i$.
Thus, 
$$|U_{i-1}[d_i]| \geq |U_{i-1}[w]| = r_{i-1}\;\; \text{ for }\;\; i \leq m ,$$
which finally gives
\begin{align*}
\sum_{v \in N[w]} c_v(w) &= \sum_{i = 1}^{m} \frac{r_{i-1} - r_i}{|U_{i-1}[d_i]|} 
                         \leq \sum_{i = 1}^{m} \frac{r_{i-1} - r_i}{|U_{i-1}[w]|} 
                         = \sum_{i = 1}^{m} \frac{r_{i-1} - r_i}{r_{i-1}} .
\end{align*}
Using the results for harmonic numbers from Lemma \ref{lemmaH}, we obtain
\begin{align*}
\sum_{v \in N[w]} c_v(w) &\leq \sum_{i = 1}^{m} \frac{r_{i-1} - r_i}{r_{i-1}} 
                         \leq \sum_{i = 1}^{m} \Big(H(r_{i-1}) - H(r_i)\Big) 
                         = H(r_0) - H(r_m) \\
                         &= H(|U_0[w]|) - H(0) 
                         = H(|N[w]|) - 0 
                         \leq \text{ln}(|N[w]|) + 1 .
\end{align*}
\qed

Now, consider a minimum cardinality $k$-tuple dominating set $D^{\mathrm{opt}}$.
Since every vertex $v$ is $k$-tuple dominated by $D^{\mathrm{opt}}$ in $G$, $|D^{\mathrm{opt}} \cap N[v]| \geq k$ and, therefore, by Lemma \ref{tdaalem2},
$$\sum_{x \in D_v} c_v(x) \;\;\leq \sum_{w \in D^{\mathrm{opt}} \cap N[v]} c_v(w) .$$
Combining the last inequality with the result from Lemma \ref{tdaalem1}, we obtain
$$|D|\, =\, \sum_{v \in V} \sum_{x \in D_v} c_v(x)\, \leq\, \sum_{v \in V} \sum_{w \in D^{\mathrm{opt}} \cap N[v]}\!\! c_v(w) .$$
For any pair of vertices $v,w \in V$, we then have
\begin{align*}
v \in V \text{ and } w \in D^{\mathrm{opt}} \cap N[v] &\iff w \in D^{\mathrm{opt}} \text{ and } w \in N[v] \\
                                         &\iff w \in D^{\mathrm{opt}} \text{ and ($v,w$ are neighbours or $v = w$)} \\
                                         &\iff w \in D^{\mathrm{opt}} \text{ and } v \in N[w]
\end{align*}

\noindent This means that we can rearrange $v$ and $w$ as summation indices in the following way: 
$$|D|\, \leq\, \sum_{v \in V} \sum_{w \in D^{\mathrm{opt}} \cap N[v]}\!\! c_v(w)\;\, = \sum_{w \in D^{\mathrm{opt}}} \sum_{v \in N[w]} c_v(w).$$
\noindent Then, by Lemma \ref{tdaalem3}, 
\begin{align*}
|D| &\leq \sum_{w \in D^{\mathrm{opt}}} \sum_{v \in N[w]} c_v(w) 
    \;\,\leq \sum_{w \in D^{\mathrm{opt}}} \text{ln}(|N[w]|) + 1 \\
    &\leq \sum_{w \in D^{\mathrm{opt}}} \text{ln}(\Delta(G) + 1) + 1 
    = \big(\text{ln}(\Delta(G) + 1) + 1\big)|D^{\mathrm{opt}}|.
\end{align*}

Thus, Algorithm \ref{BGT} is an approximation algorithm for finding a $k$-tuple dominating set in $G$ within the factor $\text{ln}(\Delta(G) + 1) + 1$ of an optimal solution. 
\qed

It is shown in \cite{kla1} that 
the $k$-tuple domination problem is not approximable within a ratio of $(1-\epsilon)\ln n$
for any $\epsilon>0$, unless all NP problems can be solved in $O(n^{\log \log n})$ time for input size $n$.
Therefore, the approximation ratio of Theorem \ref{aath2} is asymptotically the best possible (asymptotically, in the worst case) for the $k$-tuple domination problem in graphs.


\section{Approximation Algorithm for a $k$-Dominating Set}
\label{sec-k-dom}

Cicalese et al.\ \cite{cic1} proved that the minimum $k$-dominating set problem in a graph $G$ of 
maximum degree $\Delta$ can be approximated within an approximation ratio of $\ln(2\Delta)+1$ in polynomial time. 
We will show that this approximation ratio can be improved to $\ln(\Delta+k)+1$ (we are only interested in $k\le \Delta$ because if $k> \Delta$ then, trivially, $\gamma_k(G)=|V|$). 
Here, we provide a customized self-contained direct proof of a new approximation ratio of $\ln(\Delta+k)+1$ for the $k$-domination number,
improving the corresponding result and a more general proof from \cite{cic1}. 

Given a positive integer $k$, a vertex $u\in V$ is \emph{$k$-dominated} 
by a set of vertices $X\subseteq V$ if either $u\in X$ or
$|N(u)\cap X| \ge k$ for $u\not\in X$. 
Otherwise, $u$ is \emph{non-$k$-dominated} by $X$ in $G$; 
for simplicity, $u$ is also called \emph{uncovered}.
The set of vertices $k$-dominated by $X$ in $G$ is denoted by $C_k[X]$. 
We also consider the \emph{uncovered open neighbourhood} $N(v) - C_k[X]$ of a vertex $v$ in $G$, denoting its cardinality by uon$(v) = |N(v) - C_k[X]|$.

Our analysis is based on the coverage deficiency greedy heuristic described in Algorithm~\ref{CDG} below, which was originally presented in \cite{CG2021,DGCL}.
Algorithm~\ref{CDG} returns an ordered $k$-dominating set $X=\{x_1,x_2,...,x_{|X|}\}$, where each vertex $x_i$ is chosen in iteration $i$.  
In the case of $1$-domination, the basic greedy strategy described in Algorithm~\ref{alg:dom} 
will cause the maximum number of vertices to become $1$-dominated at each step. 
However, for $k$-domination with $k \geq 2$, this will not necessarily be the case,
as a non-$k$-dominated vertex may not yet be fully $k$-dominated after one of its neighbours is included in the set $D$ constructed in Algorithm \ref{alg:dom}.
Therefore, instead of only considering whether a vertex is $k$-dominated or not, 
we can also keep track of exactly how many neighbours of a non-$k$-dominated vertex are already included in the set $D$.
This provides a more subtle greedy strategy when searching for $k$-dominating sets if $k \geq 2$.

In order to do this, whenever a vertex is included in $D$,
we can place a `token' on each of its neighbours to indicate that those vertices now have one additional neighbour in $D$. 
Once $k$ tokens have been placed on a vertex, it must be $k$-dominated.
Therefore, no further tokens will be placed on it, even if another new neighbour is included in $D$.
When a vertex with less than $k$ tokens is included in $D$,
we will also place a number of tokens on it to ensure it has $k$ tokens in total, since it is now considered to be $k$-dominated.


\vspace{5mm}
\begin{algorithm}[H]
    \caption{$k$-Dominating Set (\cite{CG2021,DGCL})}
    \label{CDG}
    \KwIn{A graph $G=(V,E)$, an integer $k$, $1\le k \le \Delta$.}
    \KwOut{A $k$-dominating set $X$ of $G$.}
    \BlankLine
    
    \Begin{
    	Initialize $i=0$;\\
        Initialize $X = \emptyset$;\\
        \ForEach{$v \in V$} {
            Compute c$(v) = |N(v) \cap X| + k|\{v\} \cap X|$;\\
            Compute cd$(v) = \max\{k - \mbox{c}(v),0\}$;\\
            Compute uon$(v) = |N(v) - C_k[X]|$;
        }

        \While{$C_k[X] \neq V$} {
        		Put $i=i+1$;\\
            Select a vertex $\displaystyle u \in U = \mathrm{arg\ max}_{v \in V - X}\ (\mbox{cd}(v) + \mbox{uon}(v))$;\\
            Put $x_i=u$;\\
            Put $X = X \cup \{x_i\}$;\\
            Update the values of c$(v)$, cd$(v)$, and uon$(v)$ for $v\in V - X$;
        }
        
       \Return $X$;
    }
\end{algorithm}
\vspace{5mm}

Using this more subtle approach, we can think of $k$-domination not as trying to dominate every vertex,
but rather as trying to place $k |V(G)|$ tokens on the vertices of a graph $G$ as efficiently as possible.
In order to do this formally, a few concepts are needed.
The \textit{$k$-coverage} of a vertex $v \in V$ by a set $X \subseteq V$ 
is defined as $$\mbox{c}(v) = |N(v) \cap X| + k|\{v\} \cap X|$$
and the \textit{coverage deficiency} of $v$ by $X$ is cd$(v) = \max\{k - \mbox{c}(v),0\}$.
Here, $k|\{v\} \cap X| = k$ if $v \in X$ and $k|\{v\} \cap X| = 0$ otherwise.
Thus, c$(v) \geq k$ if and only if either $v \in X$ or $|N(v) \cap X| \geq k$,
so c$(v) \geq k$ if and only if $v$ is $k$-covered (i.e.\ $k$-dominated) by $X$.
Therefore, cd$(v) = 0$ if and only if $v$ is $k$-covered by $X$.
However, when $v$ is not $k$-covered (i.e.\ not $k$-dominated) by $X$, 
c$(v) = |N(v) \cap X|$ and cd$(v) = k - \mbox{c}(v) > 0$.
In other words, the coverage becomes a measure of how many members of $X$ already dominate $v$,
whereas the coverage deficiency measures how many neighbours still need to be added to $X$ 
in order for the resulting set to $k$-dominate $v$.

In all the cases, the coverage deficiency of a vertex indicates exactly how many tokens are still missing from the vertex to have it $k$-dominated. 
Thus, at every step, including a vertex $v$ in $X$ would add a number of tokens equal to the coverage deficiency of $v$,
plus one token for each of the non-$k$-dominated neighbours of $v$. 
A greedy heuristic (Algorithm~\ref{CDG}) that selects a vertex 
$v$ placing the largest number of tokens in an iteration behaves differently from the basic greedy strategy of Algorithm~\ref{alg:dom}.
The latter algorithm treats the inclusion of $v$ in $D$ the same as the inclusion of any other vertex, regardless of how many tokens $v$ has already received.

A generalization and adaptation of the proof of Theorem~\ref{thm-dom} for $k$-domination are presented in Theorem~\ref{aath3}. It is based on the generalization of Algorithm~\ref{alg:dom} for $k$-domination presented in Algorithm~\ref{CDG}. 
The authors of \cite{cic1} found an approximation ratio of $\text{ln}(2\Delta) + 1$ for the $k$-domination problem in a more general setting, considering a similar but much less subtle, general greedy heuristic. 

\begin{thm} \label{aath3}
The minimum $k$-dominating set problem in a graph $G$ of maximum degree $\Delta$ can be 
solved within an approximation ratio of\, $\mathrm{ln}(\Delta + k) + 1$
in polynomial time. 
\end{thm}

\pf
Similar to the proofs of Theorems \ref{thm-dom} and \ref{aath2}, let $X$ be the $k$-dominating set returned by Algorithm \ref{CDG}, and
$x_i$ be the vertex selected by the algorithm in the $i$th iteration, $i=1,2,...,|X|$.
We put $X_i = \{x_1,x_2,..., x_i\}$, $i=1,2,...,|X|$, to denote the set of the first $i$ vertices selected, with $X_0 = \emptyset$.
Here, we are no longer interested in simply tracking the uncovered closed neighbourhood of a vertex $v \in V$
because the algorithm now selects the next vertex in the iteration based on its total $k$-domination coverage contribution to the graph.
In other words, we need to track how many tokens would be placed on the vertices of $G$ if $v$ were selected as the next vertex for the $k$-dominating set.
Therefore, we define the \textit{potential coverage} of $v$ after iteration $i$, $i=0,1,...,|X|$, as
$$
U_i(v) = \begin{cases}
|N(v) - C_k[X_i]| + \max\{k - |X_i \cap N(v)|,0\}    &\text{if $v \notin X_i$}\, ; \\
0                            \quad                               \text{otherwise}.
\end{cases}
$$

Analogous to the proof of Theorem \ref{aath2}, we want to keep track of when exactly each vertex $v$ has its $k$-coverage completed,
i.e. when it attains enough tokens.
Therefore, we introduce the vector $\bar{s}(v)$ with entries $s_j(v)$ for $1 \leq j \leq k$,
where the $s_j(v)$ now represent the indices of the iterations in which $v$ receives a token.
Thus, the vector $\bar{s}(v)$ of iteration indices works in a similar way to the proof of Theorem \ref{aath2} in the case $v \notin X$.
However, when $v$ is selected by the algorithm to be included in $X$, it is possible that $v$ will add multiple tokens to itself.
Therefore, it is now possible for two different indices $s_j(v)$ in $\bar{s}(v)$ to be the same, which was not the case in the proof of Theorem \ref{aath2}.
In order to account for the same vertex being responsible for multiple tokens,
we define the collection of these vertices 
$$\bar{X_v} = (x_{s_1(v)},x_{s_2(v)},..., x_{s_k(v)})$$ 
as a vector instead of a set.
Here, $x_{s_j(v)} = x_i$ if and only if $v$ receives its $j$th token in iteration $i$.

The \textit{cost of $v$ with respect to a vertex $w \in N[v]$} is defined now to fit the new definitions as follows:
$$
c_v(w) = \begin{cases} \displaystyle
\frac{1}{U_{s_j(v)-1}(x_{s_j(v)})} &\text{if $w = x_{s_j(v)}$ for some $j \in \{1,2,...,k\}$}; \\ \displaystyle
\frac{1}{U_{s_k(v)-1}(x_{s_k(v)})} &\text{otherwise}.
\end{cases}
$$
This gives the following result on the cardinality of the $k$-dominating set created by the algorithm.
\begin{lem} \label{mdaalem1}
We have
$\displaystyle |X| = \sum_{v \in V} \sum_{j = 1}^k c_v(x_{s_j(v)})$.
\end{lem}

\pf
Note that the sum $\sum_{v \in V} \sum_{j = 1}^k c_v(x_{s_j(v)})$ now counts one cost for every token placed.
Note also that $U_{i-1}(x_i)$ is the potential coverage of $x_i$ in the iteration where $x_i$ was selected by the algorithm
and thus is equal to the number of tokens placed in iteration $i$ by $x_i$.
Therefore, 
\begin{align*}
\sum_{v \in V} \sum_{j = 1}^k c_v(x_{s_j(v)})   &= \sum_{i = 1}^{|X|} \sum_{\substack{v,j\,\text{  s.t.} \\ s_ j(v) = i}} 
                                                    c_v(x_{s_j(v)}) 
                                                = \sum_{i = 1}^{|X|} \sum_{\substack{v,j\,\text{  s.t.} \\ s_ j(v) = i}} 
                                                    c_v(x_i)
                                                = \sum_{i = 1}^{|X|} \sum_{\substack{v,j\,\text{  s.t.} \\ s_ j(v) = i}} 
                                                  \frac{1}{U_{i-1}(x_i)} \\
                                                &= \sum_{i = 1}^{|X|} U_{i-1}(x_i) \frac{1}{U_{i-1}(x_i)} 
                                                = \sum_{i = 1}^{|X|} 1 = |X|. 
\end{align*}
\up\qed

Now, let $W \subseteq N[v]$ be a subset of the closed neighbourhood of $v$, and $|W| \geq k$.
The sum of the costs of $v$ with respect to the vertices in $W$ is always at least as large as the sum of the costs corresponding to the tokens placed on $v$ during the construction of $X$. This is proved below.

\begin{lem} \label{mdaalem2}
For $W \subseteq N[v]$ with $|W| \geq k$, 
$$\displaystyle \sum_{j = 1}^k c_v(x_{s_j(v)}) \leq \sum_{w \in W} c_v(w).$$
\end{lem}

\pf
For any $j < k$,
$$U_{s_j(v)-1}(x_{s_j(v)}) \geq U_{s_j(v)-1}(x_{s_k(v)}) ,$$
since $x_{s_k(v)}$ has not yet been included in the set and, therefore, both $x_{s_j(v)}$ and $x_{s_k(v)}$ are examined by the algorithm in iteration $s_j(v)$.
Moreover, as $U_{i}(x_{s_k(v)})$ is a non-increasing, non-negative sequence with respect to $i$,
$$U_{s_j(v)-1}(x_{s_k(v)}) \geq U_{s_k(v)-1}(x_{s_k(v)}) .$$
Therefore, $c_v(x_{s_j(v)}) \leq c_v(x_{s_k(v)})$.

Let $X_v = \{x_i \mid s_j(v) = i \text{ for some } j\}$ be the set of all vertices that placed at least one token on $v$.
Note that if $v \in X_v$, then it is possible that $|X_v| < k$ in the case that $v$ places multiple tokens on itself.
We have 
\begin{equation}
\label{eq2}
|W - X_v| \geq |X_v - W| + (k - |X_v|)
\end{equation}
because $|W| \geq k = |X_v| + (k - |X_v|)$, 
and for every $w \in W - X_v$, the cost of $w$ is the same as that of $x_{s_k(v)}$.
Additionally, for every $j \geq |X_v|$, we have $x_{s_j(v)} = x_{s_k(v)}$ because if $v$ is added to $X_v$,
then $v$ immediately covers itself, and thus $v$ must be the final vertex added to $X_v$, with every remaining token also being placed by it.

Combined, these four facts give the following:
\begin{align*}
\sum_{j = 1}^k c_v(x_{s_j(v)})  &= \sum_{x \in X_v \cap W} c_v(x)\; + \sum_{x \in X_v - W} c_v(x)\; + \sum_{j = |X_v| + 1}^k c_v(x_{s_j(v)}) \\
                                &\leq \sum_{x \in X_v \cap W} c_v(x)\; + \sum_{x \in X_v - W} c_v(x_{s_k(v)})\; + \sum_{j = |X_v| + 1}^k c_v(x_{s_k(v)}) \\
                                &\leq \sum_{w \in W \cap X_v} c_v(w)\; + \sum_{w \in W - X_v} c_v(x_{s_k(v)}) 
                                    \quad\mbox{   (using (\ref{eq2}))}\\
                                &= \sum_{w \in W \cap X_v} c_v(w)\; + \sum_{w \in W - X_v} c_v(w) \\
                                &= \sum_{w \in W} c_v(w). 
\end{align*}
\qed

Next, take some vertex $w \in V$.
We wish to find a bound on 
the sum of the costs of the open neighbourhood of $w$ with respect to $w$ (i.e.\ $\sum_{v \in N(w)} c_v(w)$),
in addition to the sum of costs of $w$ with respect to the vertices that placed the $k$ tokens on $w$ 
(i.e.\ $\sum_{j = 1}^k c_w(x_{s_j(w)})$).
For this, the values of $r_i$ are now defined to be $r_i = U_i(w)$.
Recall that $U_i(w)$ was defined to be 0 after $w$ is included in $X$,
so $r_i = 0$ if either all vertices in the closed neighbourhood of $w$ are $k$-dominated, 
or when $w$ is included in $X$.
Regardless of how $r_i$ reaches 0, $m$ is still defined to be the first index such that $r_m = 0$.
The resulting bound is as follows:
\begin{lem} \label{mdaalem3}
For all $\displaystyle w \in V,$ 
$$
\sum_{v \in N(w)} c_v(w) + \sum_{j = 1}^k c_w(x_{s_j(w)}) \leq \sum_{i = 1}^{m} \frac{r_{i-1} - r_i}{r_{i-1}}
\leq \mathrm{ln}(|N(w)| + k) + 1.
$$
\end{lem}

\pf
Each vertex $v \in N(w)$ will have a cost $c_v(w) =  \frac{1}{|U_{i-1}(x_i)|}$ for some $i$, $i=1,2,...,|X|$.
If $w \notin X_v$, the index used in the cost will default to $i = s_k(v)$, given by the iteration in which $v$ is fully $k$-dominated.
However, if $w = x_j \in X_v$ for some $j$,
then the cost will instead use the index $i = j$ of the iteration in which $w$ is included in the $k$-dominating set.

Note that if $w = x_j$ for some $j$, its potential coverage must have been non-zero after iteration $j - 1$
and, therefore, $m = j$ by the definition of $r_i$.
Thus, when $i < m$ and, therefore, $w \neq x_i$, 
we have
$$
r_{i-1} - r_i = |\{v \in N(w) \mid s_k(v) = i\}| + |\{x_i\} \cap X_w|.
$$
Here, the left-hand side of the addition is the number of vertices completely $k$-covered (i.e.\ $k$-dominated) at iteration $i$
and, subsequently, the number of vertices with cost $c_v(w) = \frac{1}{U_{i-1}(x_i)}$.
The right-hand side of the addition is the number of tokens placed on $w$ at iteration $i$
and, thus, the number of tokens represented by cost $c_w(x_i) = \frac{1}{U_{i-1}(x_i)}$.

For $i = m$, there are two cases.
If $x_m \neq w$, all remaining non-$k$-dominated vertices $v \in N(w)$ are fully $k$-dominated by $x_m$ and/or the final token is placed on $w$ by $x_m$,
which means that they all have cost $c_v(w)$ or $c_w(x_m)$, which is $\frac{1}{U_{m-1}(x_m)}$.
If instead $x_m = w$, then $w \in X_v$ for all remaining non-$k$-dominated vertices $v \in N(w)$, as they were not yet $k$-dominated in iteration $m - 1$,
and $w$ will also place all missing tokens on itself immediately.
Therefore, by definition, all of them also have cost $c_v(w)=c_w(x_m) = \frac{1}{U_{m-1}(x_m)}$.

Ultimately, in every case,
$r_{i-1} - r_i$ ends up indicating the number of vertices and tokens with cost 
$\frac{1}{U_{i-1}(x_i)}$ for all $i$, $i=1,2,...,|X|$.
This gives
$$\sum_{v \in N(w)} c_v(w) + \sum_{j = 1}^k c_w(x_{s_j(w)}) = \sum_{i = 1}^{m} (r_{i-1} - r_i) \frac{1}{U_{i-1}(x_i)}.$$

Recall that if $w$ is included in the $k$-dominating set at some point, then it must be in iteration $m$.
Therefore, until at least iteration $m$, $w$ is examined by the algorithm when choosing the optimal vertex $x_i$.
Thus, 
$$U_{i-1}(x_i) \geq U_{i-1}(w) = r_{i-1}\;\; \text{ for }\;\; i \leq m ,$$
which finally gives
$$\sum_{v \in N(w)} c_v(w) + \sum_{j = 1}^k c_w(x_{s_j(w)}) = \sum_{i = 1}^{m} \frac{r_{i-1} - r_i}{U_{i-1}(x_i)} \leq \frac{r_{i-1} - r_i}{r_{i-1}}.$$
Using the results for harmonic numbers from Lemma \ref{lemmaH}, we obtain
\begin{align*}
\sum_{v \in N(w)} c_v(w) + \sum_{j = 1}^k c_w(x_{s_j(w)})   &\leq \sum_{i = 1}^{m} \frac{r_{i-1} - r_i}{r_{i-1}} 
                                                            \leq \sum_{i = 1}^{m} H(r_{i-1}) - H(r_i) 
                                                            = H(r_0) - H(r_m) \\
                                                            &= H(U_0(w)) - H(0) 
                                                            = H(|N(w)| + k) 
                                                            \leq \text{ln}(|N(w)| + k) + 1 .
\end{align*}
\qed

Now, consider a minimum size $k$-dominating set $X^{\mathrm{opt}}$.
Since every vertex $v$ is $k$-dominated by $X^{\mathrm{opt}}$, either $v \in X^{\mathrm{opt}}$ or $|X^{\mathrm{opt}} \cap N(v)| \geq k$. 
When $v \notin X^{\mathrm{opt}}$, by Lemma \ref{mdaalem2}, \up
$$\sum_{j = 1}^k c_v(x_{s_j(v)})\;\;\; \leq \sum_{w \in X^{\mathrm{opt}} \cap N(v)}\! c_v(w) .$$
Combining the last inequality with the result from Lemma \ref{mdaalem1}, we obtain \up
\begin{align*}
|X| &= \sum_{v \in V} \sum_{j = 1}^k c_v(x_{s_j(v)}) 
    \;\;= \sum_{v \in V - X^{\mathrm{opt}}} \sum_{j = 1}^k c_v(x_{s_j(v)})\;\; + \sum_{w \in X^{\mathrm{opt}}} \sum_{j = 1}^k c_w(x_{s_j(w)}) \\
    &\leq \sum_{v \in V - X^{\mathrm{opt}}} \sum_{w \in X^{\mathrm{opt}} \cap N(v)}\!\! c_v(w)\;\; + \sum_{w \in X^{\mathrm{opt}}} \sum_{j = 1}^k c_w(x_{s_j(w)}).
\end{align*}
Now, for any pair of vertices $v,w \in V$, we have
\begin{align*}
v \in V - X^{\mathrm{opt}}\ \text{ and }\ w \in X^{\mathrm{opt}} \cap N(v)  &\implies w \in X^{\mathrm{opt}} \text{ and } w \in N(v) \\
                                                        &\iff w \in X^{\mathrm{opt}} \text{ and $v,w$ are neighbours} \\
                                                        &\iff w \in X^{\mathrm{opt}} \text{ and } v \in N(w) .
\end{align*}
This means that we can rearrange $v$ and $w$ as summation indices in the following way: 
$$\sum_{v \in V - X^{\mathrm{opt}}} \sum_{w \in X^{\mathrm{opt}} \cap N(v)}\!\! c_v(w)\;\; \leq \sum_{w \in X^{\mathrm{opt}}} \sum_{v \in N(w)} c_v(w).$$
Therefore, \up\up\up
\begin{align*}
|X| &\leq \sum_{v \in V - X^{\mathrm{opt}}} \sum_{w \in X^{\mathrm{opt}} \cap N(v)}\!\! c_v(w)\; + \sum_{w \in X^{\mathrm{opt}}} \sum_{j = 1}^k c_w(x_{s_j(w)}) \\
    &\leq \sum_{w \in X^{\mathrm{opt}}} \sum_{v \in N(w)}\! c_v(w)\; + \sum_{w \in X^{\mathrm{opt}}} \sum_{j = 1}^k c_w(x_{s_j(w)}) \\
    &= \sum_{w \in X^{\mathrm{opt}}} \Big( \sum_{v \in N(w)}\! c_v(w) + \sum_{j = 1}^k c_w(x_{s_j(w)}) \Big) \\
    &\leq \sum_{w \in X^{\mathrm{opt}}} \text{ln}(|N(w)| + k) + 1\qquad\qquad \mbox{(by Lemma \ref{mdaalem3})}\\
    &\leq \sum_{w \in X^{\mathrm{opt}}} \text{ln}(\Delta(G) + k) + 1 \\
    &= \big(\text{ln}(\Delta(G) + k) + 1\big)|X^{\mathrm{opt}}|.
\end{align*}

Thus, Algorithm \ref{CDG} is an approximation algorithm for the $k$-domination problem with the approximation ratio of $\text{ln}(\Delta(G) + k) + 1$.
\qed

It is shown in \cite{cic1} that the best possible approximation ratio for the $k$-domination problem is $O(\ln n)$,
unless all NP problems can be solved in $O(n^{\log\log n})$ time for input size $n$.
Therefore, the approximation ratio of Theorem \ref{aath3} is asymptotically the best possible for the $k$-domination problem in graphs.


 \up

\end{document}